\documentclass[11pt,a4paper]{article}

\usepackage{amssymb}
\usepackage{latexsym}
\usepackage{amsmath}
\usepackage{amsfonts}
\usepackage{amsthm}
\usepackage{graphicx}
\usepackage{epstopdf}
\usepackage{lineno}
\usepackage{tikz}
\usepackage{enumerate}
\usepackage{authblk}
\usepackage{color}

\newtheorem{definition}{\bf Definition}[section]
\newtheorem{Lem}[definition]{\bf Lemma}
\newtheorem{claim}[definition]{\bf Claim}
\newtheorem{Thm}[definition]{\bf Theorem}
\newtheorem{Prop}[definition]{\bf Proposition}
\newtheorem{Cor}[definition]{\bf Corollary}

\newtheorem{Rem}[definition]{\bf Remark}
\newtheorem{ex}[definition]{\bf Examples}

\newtheorem{cons}[definition]{\bf Construction}

\graphicspath{ {./figures/} }
\newcommand{\ls}{\ell}
\newcommand{\diam}{\textrm{diam}}
\newcommand{\up}{\frac{n+\ls}{4}}

\allowdisplaybreaks

\title{Upper bounds on the $k$-isolation number}

\author[1]{Peter Borg}
\author[2]{Magdalena Lema\'nska}
\author[3]{Merc\`e Mora}
\author[4]{Mar\'ia Jos\'e Souto-Salorio}
\affil[1]{University of Malta, Malta

{\tt peter.borg@um.edu.mt}}

\affil[2]{Gda\'nsk University of Technology, Poland

{\tt magdalena.lemanska@pg.edu.pl}}

\affil[3]{
Universitat Polit\`ecnica de Catalunya, 
	Spain 
 
	{\tt merce.mora@upc.edu}}

\affil[4]{
Universidade da Coru\~na, Spain
		
	{\tt maria.souto.salorio@udc.es}}

\begin{document}

\date{}
\maketitle

\begin{abstract}
The isolation number of a graph $G$ (also called the vertex-edge domination number of $G$), denoted by $\iota(G)$, is the size of a smallest subset $D$ of the vertex set $V(G)$ of $G$ such that $G-N[D]$ (the graph obtained by deleting the closed neighbourhood $N[D]$ of $D$ from $G$) has no edges. For $k \geq 1$, the $k$-isolation number of $G$ is the size of a smallest subset $D$ of $V(G)$ such that the maximum degree of $G-N[D]$ is at most $k-1$. Thus, $\iota_1(G) = \iota(G)$. Let $n$ and $\ell$ be the number of vertices and the number of leaves of $G$, respectively. We show that if $n \geq 3$ and $G$ is connected, then $\iota_k(G) \leq \frac{n - \ell}{2}$. We also show that if $G$ is a tree $T$, then $\iota(T) \leq \frac{n + \ell}{4}$ and $\iota_k(T) \leq \frac{n + \ell}{2k+1}$ for $k \geq 2$. These bounds together improve the inequality $\iota_k(T) \leq \frac{n}{k+2}$ of Caro and Hansberg except that their inequality is better if $k \geq 2$ and $\frac{k-1}{k+2}n < \ell < \frac{k}{k+2}n$. Each of the new bounds is attainable if it is an integer. For each of them, we characterize all the graphs that attain it.
\end{abstract}
 
\section{Introduction} \label{Intro section}

Domination of graphs attracted much attention since it was introduced in the fifties. This was motivated by chessboard problems, among others (see, for example, the book \cite{haynes} and references therein). A set $D$ of vertices of a graph $G$ is {\it dominating} if every vertex not in $D$ has at least one neighbour in $D$.  
 For every graph $G=(V,E)$ and $D\subseteq V$, let $N[D]$ denote the set of vertices of $D$ and all their neighbours.
With this terminology, $D$ is a dominating set of $G$  if and only if  $V=N[D]$.
The concept of isolation arises by relaxing this condition and was introduced by Caro and Hansberg in \cite{adriana}.
 Concretely, let $\mathcal{F}$ be a set of graphs. 
 We say that $D$ is {\it $\mathcal{F}$-isolating} if no subgraph of $G-N[D]$ is a copy of a graph in $\mathcal{F}$. The \emph{$\mathcal{F}$-isolation number} of $G$, denoted by  $\iota(G, \mathcal{F})$, is the size of a smallest $\mathcal{F}$-isolating set of $G$. If $\mathcal{F} = \{F\}$, then we may replace $\mathcal{F}$ by $F$ in the defined terms and notation. If $\mathcal{F} = \{K_2\}$, then the terms $\mathcal{F}$-isolating set and $\mathcal{F}$-isolation number, and the notation $\iota(G,\mathcal{F})$, are abbreviated to isolating set, isolation number, and $\iota(G)$, respectively.  An isolating set of $G$ of size $\iota(G)$ will be called a \emph{minimum isolating set} of $G$.  

The size of a smallest dominating set of $G$ is denoted by $\gamma(G)$. Note that a subset $D$ of $V(G)$ is dominating if and only if it is $K_1$-isolating, that is, $G-N[D]$ has no vertices. Thus, $\gamma(G) = \iota(G,K_1)$. The following classical result of Ore is one of the earliest domination results.  

\begin{Thm}\label{thm:Ore} \cite{Ore} If $G$ is a connected $n$-vertex graph that is not a copy of $K_1$, then $\gamma (G)\le n/2$. Moreover, the bound is sharp.
\end{Thm}
\noindent
The graphs attaining the bound were determined independently by Payan and Xuong \cite{PX} and by Fink, Jacobson, Kinch and Roberts \cite{FJKR} (see also \cite{haynes}).

Note that a subset $D$ of $V(G)$ is isolating if and only if $V-N[D]$ is independent, that is, $G-N[D]$ has no edges. Isolation was already introduced as \emph{vertex-edge domination} in \cite{peters}. The following upper bound on the isolation number has been proved independently by Caro and Hansberg \cite{adriana} and by \.Zyli\'nski \cite{pz}.

\begin{Thm}\label{thm:caro} \cite{adriana,pz} If $G$ is a connected $n$-vertex graph that is not a copy of $K_2$ or $C_5$, then $\iota (G)\le n/3$. Moreover, the bound is sharp.
\end{Thm}
\noindent
The graphs attaining the bound were partially determined in \cite{lmm2024} by the second, third and fourth author of the present paper, and fully determined in \cite{boyergoddard} by Boyer and Goddard.

Generalizing Theorems~\ref{thm:Ore} and \ref{thm:caro}, the first author of the present paper, Fenech and Kaemawichanurat \cite{borgcliques1} showed that for any connected $n$-vertex graph $G$ and $k \geq 1$, $\iota(G, K_k) \leq n/(k+1)$ unless $G \simeq K_k$ or $k=2$ and $G \simeq C_5$. For $k = 3$, this is given by the cycle isolation bound in \cite{borgcycles}, which is extended in \cite{borgconnected}. These bounds are sharp and settled two problems in \cite{adriana}. In \cite{borgcliques2}, an upper bound in terms of the number of edges is given, and the graphs attaining it are determined. 
\vspace{1mm}

In this paper, we give sharp upper bounds on $\iota(G, K_{1,k})$ for $k \geq 1$, where $K_{1,k}$ is a star with $k$ edges. A copy of $K_{1,k}$ is called a \emph{$k$-star}. A $k$-star that is a subgraph of $G$ will be called a \emph{$k$-star of $G$}. From now on, we denote $\iota(G, K_{1,k})$ by $\iota_k(G)$. Notice that, with this notation, $\iota_1(G) = \iota(G)$. 
A $K_{1,k}$-isolating set will be called a $k$-\emph{isolating set}, and the $K_{1,k}$-isolation number will be called the \emph{$k$-isolation number}. A $k$-isolating set of $G$ of size $\iota_k(G)$ will be called a \emph{minimum $k$-isolating set} of $G$.

The \emph{maximum degree} of $G$, denoted by $\Delta (G)$, is the maximum of the degrees of the vertices of $G$. We simply write $\Delta$ if the graph is clear from the context. Notice that a subset $D$ of $V(G)$ is a $k$-isolating set of $G$ if and only if $\Delta(G-N[D]) \le k-1$. Hence, $\iota_k(G)>0$ if and only if $\Delta(G) \ge k$.

We denote by $L(G)$ the set of leaves of $G$, and by $\ls(G)$ the number $|L(G)|$ of leaves of $G$. We call $\ls(G)$ the \emph{leaf order} of $G$. We abbreviate $\ls(G)$ to $\ls$ if the graph is clear from the context.

In Section~\ref{sec:upperboundsisolation}, we prove the following result. 

\begin{Thm}\label{thm:orderminusleaves_k} 
If  $G$ is a connected graph of order $n \ge 3$ and leaf order $\ell$, and $G$ is not a star, then for any $k \geq 1$, 

\begin{equation} \iota_k(G)\le \frac{n-\ell}{2}. \label{iota_k bound}
\end{equation}
Moreover,
for any $k \geq 1$, $r \geq 1$ and $n \ge \max \{(k+2)r, 4\}$, there exists an $n$-vertex graph $G$ 
such that $\iota_k(G) = \frac{n-\ell}{2} = r$.
\end{Thm}

\noindent Theorem~\ref{thm:orderminusleaves_k} is surprising in at least two ways. Firstly, the bound 
is independent of $k$ but attainable for any $k$ in the strong manner described in the second part of the theorem. Secondly, (\ref{iota_k bound}) is equivalent to Theorem~\ref{thm:orderminusleaves}. Indeed, Theorem~\ref{thm:orderminusleaves} is the case $k = 1$ of (\ref{iota_k bound}), and 
Theorem~\ref{thm:orderminusleaves} implies (\ref{iota_k bound}). Note that (\ref{iota_k bound}) improves Theorem~\ref{thm:caro} for $\ell > n/3$.

In their seminal paper \cite{adriana}, Caro and Hansberg also established the following generalization of Theorem~\ref{thm:caro} for trees.

\begin{Thm}\label{thm:caro_trees} 
\cite{adriana}
If $G$ is an $n$-vertex tree that is not a $k$-star, then $\iota_k(G)\le n/(k+2)$. Moreover, the bound is sharp.
\end{Thm}

In Section~\ref{sec:upperboundsisolationtree}, the bound in (\ref{iota_k bound}) is improved for $k = 1$ and any tree with $\ell < n/3$. The two bounds together improve the one in Theorem~\ref{thm:caro_trees} for $k = 1$ (also given by Theorem~\ref{thm:caro}) except that the three bounds are equal when $\ell = n/3$. Section~\ref{sec:upperboundsstarisolation} provides a sharp upper bound on $\iota_k(T)$ for any $k \ge 2$ and any tree $T$ in terms of $n$ and $\ell$. Surprisingly, while this bound also holds for $k = 1$, it is inferior to the one in Section~\ref{sec:upperboundsisolationtree}, and therefore not sharp, for $k = 1$. The bounds in Sections~\ref{sec:upperboundsisolation} and \ref{sec:upperboundsstarisolation} together improve Theorem~\ref{thm:caro_trees} except when $\frac{k-1}{k+2}n \le \ell \le \frac{k}{k+2}n$. The graphs attaining the above-described new bounds are characterized.
\vspace{2mm} 

The following result, proved in Section~\ref{sec:upperboundsstarisolation}, is the improvement of Theorem~\ref{thm:caro_trees} that the isolation bounds we establish in the subsequent sections culminate in as far as trees are concerned.

\begin{Thm} \label{thm:generaltreebound}
Let $T$ be a non-star tree of order $n$ and leaf order $\ell$. \\
(a)  
\[ \iota(T)  \leq \left\{\begin{array}{ll}
                     \frac{n + \ell}{4} < \frac{n}{3}  \quad  & \mbox{if}~~\ell < \frac{n}{3},\\\\
                     \frac{n + \ell}{4} = \frac{n - \ell}{2} = \frac{n}{3}   \quad  & \mbox{if}~~\ell = \frac{n}{3},\\\\
                     \frac{n - \ell}{2} < \frac{n}{3}  \quad  & \mbox{if}~~\ell > \frac{n}{3}.
               \end{array}\right. \]
(b) If $k \geq 2$, then
\[ \iota_{k}(T)  \leq \left\{\begin{array}{ll}
                     \frac{n + \ell}{2k+1} < \frac{n}{k+2}  \quad  & \mbox{if}~~\ell < \frac{k-1}{k+2}n,\\\\
                     \frac{n + \ell}{2k+1} = \frac{n}{k+2}  \quad  & \mbox{if}~~\ell = \frac{k-1}{k+2}n,\\\\
                     \frac{n}{k+2}  \quad  & \mbox{if}~~\frac{k-1}{k+2}n < \ell < \frac{k}{k+2}n,\\\\
                     \frac{n - \ell}{2} = \frac{n}{k+2}   \quad  & \mbox{if}~~\ell = \frac{k}{k+2}n,\\\\
                     \frac{n - \ell}{2} < \frac{n}{k+2}  \quad  & \mbox{if}~~\ell > \frac{k}{k+2}n.
               \end{array}\right. \]
Moreover, the bounds are sharp.
\end{Thm}

We conclude this section with some further terminology, notation and known results.

\vspace{2mm}

A \emph{support vertex} is a vertex with at least one leaf adjacent to it.  A \emph{strong support vertex} is a support vertex with at least two adjacent leaves. The \emph{diameter} of $G$, denoted by $\diam (G)$, is the maximum of the distances in $G$, that is, $\diam (G) = \max\{{\rm dist}_G(u,v) \colon u, v \in V(G)\}$, where ${\rm dist}_G(u,v)$ denotes the distance between $u$ and $v$. We often use the well-known fact that the diameter of a tree $T$ is the length of a longest path of $T$, and we call such a path a \emph{diametral path} of $T$.
For undefined terms, we refer the reader to \cite{CHL,West}.

\vspace{3mm}

\color{black}
The known equation (obtained from the handshaking lemma) \begin{equation} \ell =2+\sum_{i=3}^{\Delta} n_i (i-2) \label{lformula}
\end{equation}
for the number $\ell$ of leaves of a tree with maximum degree $\Delta \geq 3$ in terms of $n_3,\dots, n_{\Delta}$, where $n_i$ denotes the number of vertices of degree $i$,
will be used often.

Next, we prove two general properties of isolating sets.

\begin{Lem}\label{lem:isolatingnoleaves} If $k \geq 1$ and $G$ is a connected $n$-vertex graph with $n \geq 3$, then $G$ has a minimum $k$-isolating set with no leaves.
\end{Lem}
\begin{proof}
Since $G$ is connected and $n \geq 3$, no support vertex of $G$ is a leaf. Let $D$ be a minimum $k$-isolating set of $G$. 
Suppose that $D$ contains leaves. Obviously, the set $D'$ obtained by replacing the leaves of $D$ by their support vertices is also a $k$-isolating set of $G$. Since $|D| = \iota_k(G) \le |D'| \le |D|$, $D'$ is a minimum $k$-isolating set of $G$.
\end{proof}

\begin{Lem}\label{isolatingnoP2} If $G$ is a connected $n$-vertex graph with $n \geq 5$, then $G$ has a minimum isolating set with no leaves and no support vertices of degree $2$.
\end{Lem}
\begin{proof}
Suppose that $v$ is a support vertex of degree $2$ and $u$ is a leaf adjacent to $v$. Since $G$ is connected and $n \geq 5$, the other neighbour $w$ of $v$ is neither a leaf nor a support vertex of degree $2$ (otherwise, $V(G)$ consists of $u, v, w$ and a leaf $x \neq u$ adjacent to $w$). Thus, the non-leaf neighbour of a support vertex of degree $2$ is neither a leaf nor a support vertex of degree $2$. 

By Lemma~\ref{lem:isolatingnoleaves}, $G$ has a a minimum isolating set $D$ containing no leaves. Let $D'$ be the set obtained from $D$ by replacing each support vertex of degree $2$ by its non-leaf neighbour. Then, $D'$ is also an isolating set of $G$. Since $|D|=\iota (G)\le |D'|\le |D|$, the result follows.
\end{proof}


\section{A general upper bound on the isolation number}\label{sec:upperboundsisolation}

We provide an upper bound on the isolation number that is better than Theorem~\ref{thm:caro} for any graph whose leaf order is at least one third of its order. 

\begin{Thm}\label{thm:orderminusleaves} 
If $G$ is a connected graph of order $n \ge 3$ 
and leaf order $\ell$, and $G$ is not a star,
then

$$\iota(G)\le \frac{n-\ell}{2}.$$    
\end{Thm}

\begin{proof}
Let $G'$ be the graph obtained by removing the leaves of $G$, and let $D$ be a minimum dominating set of $G'$. Then,
$G'$ is a connected graph of order at least 2 and $D$ is trivially an isolating set of $G$  (see Figure~\ref{fig:demoorderminusleaves}). Hence, 
\begin{equation}\label{eqdom}
\iota (G)\le \gamma(G')\le \frac{\vert V(G')\vert}2=\frac{n-\ell}2    
\end{equation} 
with the second inequality given by Theorem~\ref{thm:Ore}.
\end{proof}

\begin{figure}[ht!]
\centering
 \includegraphics[width=0.45\textwidth]{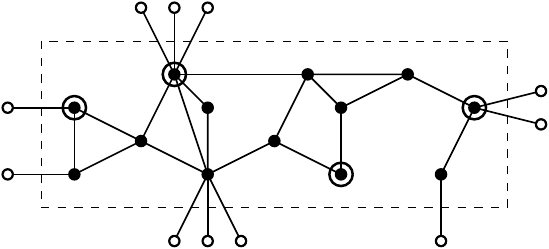}
\caption{A graph $G$: white vertices are in $L(G)$ and black vertices in $G-L(G)$. Circled vertices form a dominating set of $G-L(G)$ and, consequently, an isolating set of $G$.}
\label{fig:demoorderminusleaves}
\end{figure} 

We now show that Theorem~\ref{thm:orderminusleaves} surprisingly generalizes to Theorem~\ref{thm:orderminusleaves_k}. 

\begin{proof}[Proof of Theorem~\ref{thm:orderminusleaves_k}]
Trivially, $\iota_k(G) \leq \iota(G)$, so $\iota_k(G)\le \frac{n-\ell}{2}$ by Theorem~\ref{thm:orderminusleaves}. We now prove the second part. 

Suppose first that $k\ge 1$ and $r\ge 2$. Let $H_0 \simeq P_r$. Thus, $E(H_0) = \{v_i v_{i+1} \colon i \in [r-1]\}$, where $\{v_1, \dots, v_r\} = V(H_0)$. Let $H_1, \dots, H_r$ be $k$-stars such that $H_0, H_1, \dots, H_r$ are pairwise vertex-disjoint. For each $i \in [r]$, 
let $w_i \in V(H_i)$ with $N_{H_i}[w_i] = V(H_i)$. Let $H$ be the graph with $V(H) = \bigcup_{i=0}^r V(H_i)$ and $E(H) = \{v_iw_i \colon i \in [r]\} \cup \bigcup_{i=0}^r E(H_i)$. We have $|V(H)| = (k+2)r$. If $n = (k+2)r$, then let $G = H$. If $n > (k+2)r$, then let $G$ be the graph obtained from $H$ by adding $n - (k+2)r$ vertices to $H$ and making them adjacent to $w_r$. Thus, $|V(G)| = n$, $L(G) = \bigcup_{i=1}^r N_G(w_i) \setminus \{v_i\}$, and $n - \ell = |V(G) \setminus L(G)| = |\bigcup_{i=1}^r \{v_i, w_i\}| = 2r$. Clearly, if $I$ is a $k$-isolating set of $G$, then $I \cap (V(H_i) \cup \{v_i\})$ for each $i \in [r]$. Thus, $\iota_k(G) \geq r$. Since $\{w_i \colon i \in [r]\}$ is a $k$-isolating set of $G$, $\iota_k(G) = r$.  

Now suppose $k\ge 1$, $r=1$  and $n\ge \max\{k+2,4\}$. Let $G$ be an $n$-vertex graph obtained by attaching a leaf to one of the leaves of an $(n-2)$-star.  Then, $\ell=n-2$ and 
$\iota_k(G)=1=\frac{n-\ell}2$.
\end{proof}

The graphs attaining the bound in Theorem~\ref{thm:orderminusleaves_k} can be characterized from the graphs attaining the upper bound in Theorem~\ref{thm:Ore}. Let $\mathcal{G}$ be the family of connected graphs with $\gamma(G)={\vert V(G) \vert}/2$. Let $\mathcal{G}'$ be the family of graphs $G$ such that $G$ is obtained by attaching exactly one leaf to every vertex of a connected graph $H$ of order at least 2 (that is, $G$ is the {\it corona product} of $H$ and $K_1$). In the above-mentioned papers \cite{PX, FJKR}, it was shown that

 $$\mathcal{G} = \{G \colon G \simeq C_4\textrm{ or } G\simeq K_2\} \cup \mathcal{G}'.$$
 
\begin{Thm}\label{thm:charorderminusleaves}
 Let $G$ be a connected graph of order $n\ge 3$.
Then, 
$\iota_k(G) = \frac{n-\ell}{2}$
if and only if $G$ is obtained in one of the following ways:
\begin{enumerate}[(a)]
\item  by attaching a non-negative number of leaves to each vertex of $K_2$ with at least $k-1$ leaves  attached to one of them;
\item by attaching at least $k$ leaves to each vertex of a $4$-cycle;
\item by attaching a non-negative number of leaves to each vertex of some $G' \in \mathcal{G'}$ in such a way that at least $k$ leaves are attached to each leaf of $G'$.   
\end{enumerate}
\end{Thm}
\begin{figure}[ht!]
\centering
\includegraphics[width=0.8\textwidth]{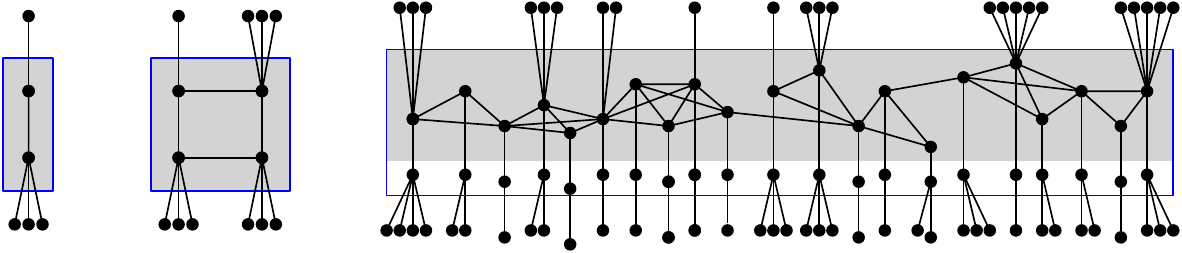}
\caption{Graphs $G$ with $\iota(G)=\frac{n-\ell}2$.} 
\label{fig:exampleorderminusleaves}
\end{figure} 
\begin{proof}
Suppose $\iota_k(G) = \frac{n-\ell}2$. Since $\iota_k(G) \leq \iota(G)$, $\iota(G) = \frac{n-\ell}2$ by Theorem~\ref{thm:orderminusleaves}. 
Let $G'$ be the graph obtained by removing all the leaves from $G$. Then, all the inequalities in (\ref{eqdom}) must be equalities. 
Thus, $G'\in \mathcal{G}$. 
Suppose first $G'\simeq K_2$. Then, $n=2+\ell$  and $\iota_k(G)=1$. Therefore, there is at least one leaf attached at each vertex of $K_2$ (otherwise, $\ell=n-1$) and at least $k-1$ attached at one of the vertices of $K_2$ (otherwise, $G$ has maximum degree less than $k$, implying that $\iota_k(G)=0$).
Now, suppose $G' \simeq C_4$. Then, $n=4+\ell$. Since $\iota_k(G) = \frac{n-\ell}{2}$, $\iota_k(G) = \frac{4}{2} = 2$. Thus, $G$ is obtained by attaching at least $k$ leaves to each vertex of $G'$ (as otherwise $\iota_k(G) = 1$). 
Finally, suppose $G' \in \mathcal{G'}$. Let $r = \frac{n-\ell}{2}$. We have that $G'$ has exactly $2r$ distinct vertices $v_1, \dots, v_r, w_1, \dots, w_r$ such that $G'[\{v_1, \dots, v_r\}]$ is a connected graph $H'$ and, for each $i \in [r]$, $w_i$ is a leaf of $G'$ adjacent to $v_i$. Consider any $i \in [r]$. By construction, $w_i$ is not a leaf of $G$, so $w_i$ is a support vertex of $G$. Suppose that $w_i$ has less than $k$ leaves adjacent to it in $G$. Since $H'$ is connected, we clearly obtain that $V(H') \setminus \{v_i\}$ is a $k$-isolating set of size $r-1$, but this contradicts $\iota_k(G) = r$. Thus, $w_i$ has at least $k$ leaves adjacent to it in $G$.

We now prove the converse. 

If (a) holds, then clearly $\iota_k(G)=1=\frac{n-\ell}2$.
If (b) holds, then clearly $\iota_k(G) \geq 2$ and any set of two vertices of the $4$-cycle is a smallest $k$-isolating set of $G$, so $\iota_k(G) = \frac{4}{2} = \frac{n-\ell}{2}$. 
Finally, suppose that (c) holds (for $k=1,$ see Figure~\ref{fig:exampleorderminusleaves}). By Lemma~\ref{lem:isolatingnoleaves}, $G$ has a minimum $k$-isolating set $D$ with no leaves, so $D \subseteq V(G')$. 
Each leaf $z$ of $G'$ is a support vertex of $G$ with $|N_G(z) \cap L(G)| \geq k$ (and hence $G[N_G[z]]$ contains a $k$-star), so at least one of $z$ and its support vertex in $G'$ is in $D$. Since no two leaves of $G'$ have the same support vertex, $|D| \geq |L(G')|$. Since $G' \in \mathcal{G}'$, $|V(G')| = 2|L(G')|$. Thus, we have 

$$\iota_k(G) = \vert D\vert \ge |L(G')| = \frac{|V(G')|}2 = \frac{n-\ell}2.$$
Together with Theorem~\ref{thm:orderminusleaves_k}, this gives us $\iota_k(G) = \frac{n-\ell}2$.
\end{proof}


\section{An upper bound on the isolation number of a tree} \label{sec:upperboundsisolationtree}
 
 From now on, we focus on trees. We use the well-known fact, resulting from (\ref{lformula}), that every tree 
 of order $n\geq 2$
has at least $2$ leaves.
\begin{Thm}\label{th:orderplusleaves}  If $T$ is a tree of order $n$ and leaf order $\ell$, then 
$$\iota(T) \leq \frac{n+\ell}{4}.$$ 
Moreover, for any integer $r \geq 1$, there exist two integers $n$ and $\ell$, and a tree $T$ of order $n$ and leaf order $\ell$, such that $\iota(T) = \frac{n + \ell}{4} = r$.
\end{Thm}

\begin{proof} We use induction on $n$. If $n = 1$, then $\iota(T) = 0 < \frac{n + \ell}4$. If $n = 2$, then $T \simeq K_2$, so $\ell = 2$ and $\iota(T) = 1 = \frac{n + \ell}4$. Suppose $n\ge 3$. 
Let  $d=\diam (T)$.
Then, there exist $d+1$ distinct members $u_0, u_1, \dots, u_d$ of $V(T)$ such that $T[\{u_0, u_1, \dots, u_d\}]$ is a longest path $U$ of $T$. If $d = 1$, then $\{u_0\}$ is an isolating set of $T$, so $\iota(T) = 1 < \frac{5}{4} \leq \frac{n + \ell}4$. Suppose $d \geq 2$.

Consider any $i \in [d-1]$. Clearly, $T - u_iu_{i+1}$ has two components $T_i$ and $T_i'$, where $T_i$ contains $u_i$, and $T_i'$ contains $u_{i+1}$. Consequently, by the choice of $U$, the distance (in $T$) of any vertex of $T_i$ from $u_i$ is at most $i$, and the distance (in $T$) of any vertex of $T_i'$ from $u_i$ is at most $d-i$. By considering $T_1$, this particularly implies that $u_0$ is a leaf of $T$ (because if $v \in V(T_1) \setminus \{u_1\}$, then $v$ is adjacent to $u_1$, and since $T_1$ contains no cycles, $v$ is not adjacent to $u_0$). 

If $d \le 4$, then $\{u_2\}$ is an isolating set of $T$, so $\iota(T) = 1 < \frac{5}{4} \leq \frac{n+\ell}4$. Suppose $d \ge 5$. We have $\deg(u_3) \geq 2$.

Suppose $\deg_T(u_{3})\ge 3$. Let $T'= T_2'$, $n' = |V(T')|$ and $\ell' = \ell(T')$. Notice that $3 \le n'\le n-3$ and $\ell' \le \ell - 1$. Let $D'$ be a minimum isolating set of $T'$. By the induction hypothesis, $|D'| \leq \frac{n' + \ell'}{4}$. Let $D = D'\cup \{u_2\}$. By the above, $D$ is an isolating set of $T$ (see Figure~\ref{fig:demodegree3}). We have 

$$\iota(T) \leq |D|=|D'|+1\leq \frac{n'+\ell'}{4} +1\leq  \frac{n-3+\ell-1}{4} +1
= \frac{n+\ell}{4}.$$ 
    
    \begin{figure}[h]
    \centering
     \includegraphics[width=0.4\textwidth]{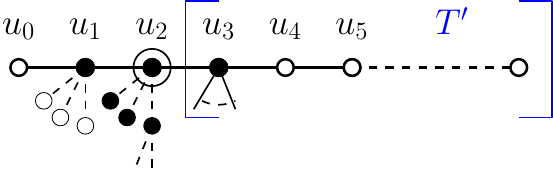}
    \caption{An isolating set of $T$ is formed by $u_2$ and an isolating set of $T'$.}
    \label{fig:demodegree3}
    \end{figure} 

Now suppose $\deg_T(u_{3}) = 2$. Let $T'= T_3'$, $n' = |V(T')|$ and $\ell' = \ell(T')$. Notice that $2 \le n'\le n-4$ and $\ell' \le \ell$. Let $D'$ be a minimum isolating set of $T'$. By the induction hypothesis, $|D'| \le \frac{n'+\ell'}{4}$. Let $D = D'\cup \{u_2\}$. By the above, $D$ is an isolating set of $T$ (see Figure~\ref{fig:demodegree2}). We have 

$$\iota(T) \leq |D| = |D'| + 1 \le \frac{n' + \ell'}{4} + 1 \leq \frac{n - 4 + \ell}{4} + 1 = \frac{n + \ell}{4}.$$ 

\begin{figure}[h]
\centering
 \includegraphics[width=0.4\textwidth]{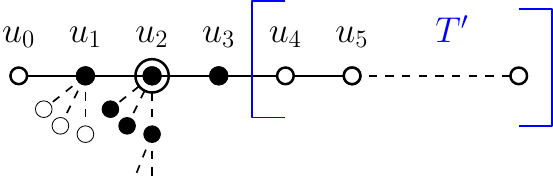}
\caption{An isolating set of $T$ is formed by $u_2$ and an isolating set of $T'$.}
\label{fig:demodegree2}
\end{figure} 

The second part of the theorem is given by Theorem~\ref{charactupperbound}.
\end{proof}

We now characterize the trees that attain the upper bound given in Theorem \ref{th:orderplusleaves}. To this aim, we introduce the following family.
\begin{cons}\label{notation:familyF}

Let $\mathcal{F}$ be the family of trees obtained in the following way. Consider copies of a path of order $3$ with one leaf labeled $a$, the other leaf labeled $c$ and the vertex of degree $2$ labeled $b$, and copies of a path of order $4$ with the leaves labeled $x$ and the vertices of degree $2$ labeled $y$, where the $3$-vertex copies and the $4$-vertex copies are pairwise vertex-disjoint.  Let $A$, $B$, $C$, $X$ and $Y$ be the sets of vertices of these copies labeled $a$, $b$, $c$, $x$ and $y$, respectively. A tree $T$ belongs to $\mathcal{F}$ if and only if $T$ can be obtained from $r \ge 2$ copies of the path of order $3$ and $s \ge 0$ copies of the path of order $4$ by adding edges with endpoints in $A\cup X$ so that no vertex in $A\cup X$ remains a leaf in $T$ (that is, the leaves of $T$ are the vertices in $C$).
\end{cons}
Notice that it is possible to construct such a tree $T$ for every $r\ge 2$ and $s\ge 0$. For example, if $s=0$, add edges between the vertices of $A$ forming a path of order $r$, and if $s\ge 1$, join the $s$ copies of $P_4$ by adding edges with endpoints in $X$ forming a path $P$ of order $4s$, add one edge joining one endpoint of $P$ with one vertex of $A$, and add edges joining the remaining vertices in $A$ with the other endpoint of $P$ (see Figure~\ref{existenceF}).

\begin{figure}[h]
\centering
 \includegraphics[width=0.9\textwidth]{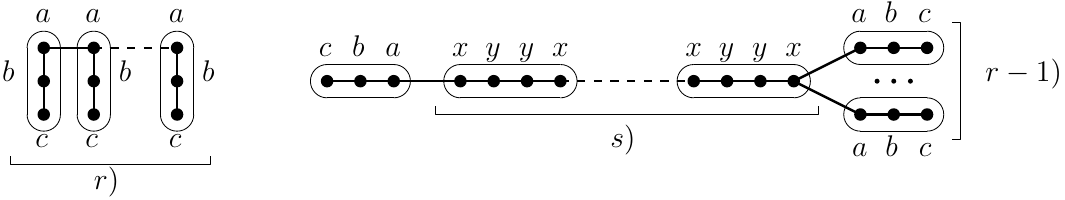}
\caption{Trees belonging to the family $\mathcal{F}$ with $r$ copies of $P_3$ (left) and with $r$ copies of $P_3$ and $s\ge 1$ copies of $P_4$.}
\label{existenceF}
\end{figure}

By definition, the trees in $\mathcal{F}$ satisfy the following properties (see an example in Figure~\ref{fig:familytreesleaves}).

\begin{figure}[h]
\centering
 \includegraphics[width=0.5\textwidth]{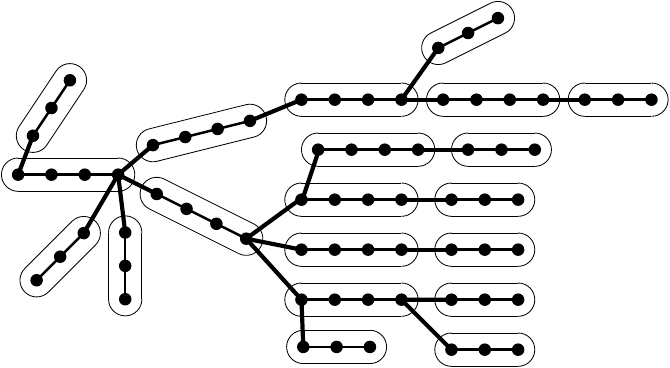}
\caption{A tree in $\mathcal{F}$ constructed from copies of $P_3$ and $P_4$.}
\label{fig:familytreesleaves}
\end{figure}

\begin{Rem}\label{propertiesF} If $T \in \mathcal{F}$ is a tree of order $n$ and leaf order $\ell$, then
\begin{enumerate}[a)]
    \item $C$ is the set of leaves of $T$; 
    \item $B$ is the set of support vertices of $T$. Moreover, if $u\in B$, then $\deg (u)=2$ and $u$ has exactly one neighbour in $A$ and one neighbour in $C$
    \item If $u\in X$, then $u$ has  exactly one neighbour in $Y$ and at least one neighbour in $X\cup A$; 
     \item If $u,v\in X$ and $uv$ is an edge of $T$, then $u$ and $v$ belong to different copies of $P_4$;
    \item If $u\in Y$, then $\deg (u)=2$  and $u$ has  exactly one neighbour in $X$ and one neighbour in $Y$;
    \item The number of copies of $P_4$ used to construct $T$ is $\frac{|X|}2$.
    Moreover, $|X|$ and $|Y|$ are equal and even;
    \item The number of copies of $P_3$ use to construct $T$ is $|A|$ and $|A|=|B|=|C|$.
    \item $n=3|A|+2|X|\ge 6$;
    \item $\displaystyle\frac {n+\ls}4=|A|+\frac{|X|}2\in \mathbb{Z}$;
    \item The vertices of $X\cup Y$ induce a forest such that all components have order a multiple $4.$ 
\end{enumerate}
    
\end{Rem}

\begin{Prop}\label{prop:isolationtreeF}
If $T\in \mathcal{F}$, then 
$\iota (T)=\frac {n+\ls}4$. 
Moreover, $T$ has a minimum isolating set formed by all the vertices in $A$ and one vertex labeled with $x$ from each copy of $P_4$.
\end{Prop}

\begin{proof} 
Let $T\in \mathcal {F}$. Since all support vertices of $T$ have degree $2$, we know that there is a minimum isolating set containing all the vertices in $A$ and no vertex from $B\cup C$. Assume that $D$ is a minimum isolating set such that $A\subseteq D$ and $D\cap (B\cup C)=\emptyset$. 
Since the set $N[A]$ contains no vertex from $Y$, $D$ must contain at least one vertex from each copy of $P_4$, otherwise, the two vertices in $Y$ from a copy of $P_4$ with no vertex in $D$ are not isolated. Hence, by Remark~\ref{propertiesF} $i)$, $\iota (T)\ge |A|+\frac{|X|}2=\frac {n+\ls}4$. By Theorem~\ref{th:orderplusleaves}, $\iota (T)=\frac {n+\ls}4$.

If $T\in \mathcal {F}$, a minimum isolating set can be constructed as follows. Choose a vertex $r\in A\cup X$ as the root of $T$. Let $X_0$ be the set of vertices containing the vertex in $X$ closer to the root $r$ from each copy of $P_4$. We claim that the set $D_0=A\cup X_0$ is an isolating set of $T$. Indeed, let $Y_0$ be the set of vertices containing the vertex in $Y$ closer to the root $r$ from each copy of $P_4$ and set $Y_1=Y\setminus Y_0$. Let $X_1=X\setminus X_0$.
Notice first that $A\cup B\cup X_0\cup Y_0\subseteq N[D_0]$. 
Now, if $v\in C$, then $v$ is isolated in $T-N[D_0]$, because its only neighbour is a vertex $z\in B\subseteq N[D_0]$.
If $v\in X_1$, then $v$ has at least one neighbour in $A\cup X_0$, implying that $v\in N[D_0]$.
Thus, $X_1\subseteq N[D_0]$.
If $v\in Y_1$, then all its neighbours belong to $Y_0\cup X_1\subseteq N[D_0]$. Hence, $v$ is isolated in $T-N[D_0]$.
Therefore, $D_0$ is an isolating set of $T$.
Moreover, $|D_0|=|A|+\frac{|X|}2=\frac{n+\ell}4$. Hence, $D_0$ is a minimum isolating set.
\end{proof}

Theorem~\ref{charactupperbound} will show that, apart from $K_2$, 
the only trees 
satisfying this upper bound 
are the trees in $\mathcal{F}$. 
We begin by proving some cases in which a tree does not attain the upper bound. 
\begin{Lem}\label{diam4} If $T$ is a tree of order $n\ge 3$ and diameter 2, 3 or 4, 
 then $\iota(T)=1< \frac{n+\ls}4$. 
 \end{Lem}
 \begin{proof}  
 Since the radius of $T$ is at most 2, any set formed by a central vertex of $T$ is isolating. Hence, $\iota (T)=1$. 
 Besides, $\frac{n+\ell}4\ge \frac{3+2}4>1$.
 \end{proof}
 
\begin{Lem}\label{nostrong}
If a tree $T$ of order $n\ge 3$ and leaf order $\ls$ has a strong support vertex, then  $\iota (T)<\frac{n+\ls}4$.
\end{Lem}
\begin{proof}
Suppose  $u$ is a strong support vertex of $T$ and let $x$ and $y$ be two leaves adjacent to $u$. Let $D$ be a minimum isolating set of $T'=T-\{x\}$ with no leaves, that exists by Lemma~{\ref{lem:isolatingnoleaves}}.
Obviously, $D$ is also an isolating set of $T$.
Since $T'$ has order $n-1$ and $\ls-1$ leaves, by Theorem~\ref{th:orderplusleaves}

$$\iota (T)\le |D|=\iota (T')\le \frac{n-1+\ls-1}4<\frac{n+\ls}4.$$
\end{proof}

\begin{Lem}\label{pending_p4}
Let $T$ be a tree of order $n\ge 5$ and leaf order $\ls$. If $T$ can be obtained by adding an edge between a non-leaf vertex $v$ of a path of order 4 and a vertex of a tree $T'$ of order $n-4$, then $\iota (T)<\frac{n+\ls}4$.
\end{Lem}
\begin{proof}
If $n=5$, then $\{v\}$ is an isolating set of $T$ and $\frac{n+\ell}4\ge \frac{5+2}4>1=\iota (T)$.
If $n\ge 6$, then $T'$ has order $n-4\ge 2$ and at most $\ls-1$ leaves.
Let $D'$ be a minimum isolating set of $T'$. Then $D'\cup \{ v \}$ is an isolating set of $T$ and, by Theorem~\ref{th:orderplusleaves}

 $$ \iota (T)\le \iota (T')+1\le \frac{(n-4)+(\ls -1)}4+1<\up.$$
 \end{proof}

\begin{Lem}\label{pending_p5}
Let $T$ be a tree of order $n\ge 6$ and leaf order $\ls$. If $T$ can be obtained by adding an edge between a support vertex of a path $P$ of order 5 and a vertex of a tree $T'$ of order $n-5$, then $\iota (T)<\frac{n+\ls}4$.
\end{Lem}
\begin{proof}
Suppose that $v$ is the only vertex of $P$ that is neither a leaf nor a support vertex. If $n=6$, then $\{v\}$ is an isolating set of $T$ and $\frac{n+\ell}4\ge \frac{6+2}4>1=\iota (T)$. Suppose $n\ge 7$. Then, $T'$ is of order at least 2 and has at most $\ls -1$ leaves. Let
$D'$ be a minimum isolating set of $T'$. Then, $D'\cup \{ v \}$ is an isolating set of $T$. Since $T'$ is of order $n-5\ge 2$ and has at most $\ls-1$ leaves, Theorem~\ref{th:orderplusleaves} yields

 $$ \iota (T)\le \iota (T')+1\le \frac{(n-5)+(\ls -1)}4+1<\up.$$
 \end{proof}
 
\begin{Thm}\label{charactupperbound}
If a tree $T$ of order $n\ge 6$ and leaf order $\ls$ satisfies $\iota (T)=\frac {n+\ls}4$, then $T\in \mathcal{F}$.
\end{Thm}

\begin{proof} 
The proof is by induction on $n$.
Let $T$ be a tree of order $n\ge 6$ such that $\iota (T)=\frac {n+\ls}4$. Let $d$ be the diameter of $T$.
It can be checked that the only tree of order 6 attaining the upper bound is $P_6$ and it is in $\mathcal{F}$, because it can be obtained from two copies of $P_3$ by adding an edge between the vertices in $A$. Hence, the statement is true for $n=6$.

Now suppose that $T$ is a tree of order $n$, $n\ge 7$, with $\ls $ leaves such that $\iota (T)=\frac{n+\ls}4$. By Lemma~\ref{diam4}, the diameter of $T$ is at least 4. Let $P=(u_0,u_1,\dots ,u_d)$ be a diametral path of $T$. By the choice of $P$ and by Lemma~\ref{nostrong}, $\deg_T (u_0)=1$,  $\deg (u_1)=2$ and the set $N(u_2)\setminus \{u_1,u_3\}$ contains at most one leaf and  $k$ vertices of degree $2$ that are support vertices, for some $k\ge 0$.
By Lemma~\ref{pending_p4}, it is not possible to have $k=0$ and a leaf adjacent to $u_2$. Hence, either $\deg (u_2)=2$ or 
$\deg (u_2)\ge 3$ with $k\ge 1$.  

If $\deg (u_2)\ge 3$ with $k\ge 1$, let $T'_2$ be the connected component of the $T-u_2u_3$ containing the vertex $u_3$. If $D'$ is a minimum isolating set of $T'_2$, then $D=D'\cup \{u_2\}$ is an isolating set of $T$. If $T'_2$ has order $n'$ and $\ls'$ leaves, then  $n'\le n-5$ and $\ls '\le \ls -1$. Thus, 

{\small
$$\up =\iota (T)\le |D'|+1= \iota (T'_2)+1\le \frac {n'+\ls'}4+1\le \frac{(n-5)+(\ls -1)}4+1<\up,$$}
a contradiction. Therefore, $\deg (u_2)=2$. 
Next, we distinguish cases according to the degree of $u_3$. 

{\bf Case 1:} $\deg (u_3)=2$. Notice that $\deg (u_4)\ge 2$, because $n\ge 7$.   
Let $T'_3=T-\{u_0,u_1,u_2,u_3\}$ and let $D'$ be a minimum isolating set of $T'_3$. Then, $D'\cup \{u_2\}$ is an isolating set of $T$. 
If $\deg (u_4)\ge 3$, 
 then $T'_3$ is a tree of order $n-4$ with $\ls -1$ leaves.
Thus,

$$\up =\iota (T)\le |D|=|D'|+1= \iota (T'_3)+1\le \frac {(n-4)+(\ls -1)}4+1<\up,$$
a contradiction.
Hence,  $\deg (u_4)=2$
and $T'_3$ is a tree 
with $\ls$ leaves. Moreover,
     
$$\up =\iota (T)\le |D|=|D'|+1= \iota (T'_3)+1\le \frac {(n-4)+\ls}4+1=\up,$$

implying that all inequalities in the preceding expression must be equalities. Hence, 

$$\iota (T'_3)=\frac {(n-4)+\ls}4$$    
and, by the induction hypothesis, $T'_3\in \mathcal{F}$. Consider the construction of $T'_3$ from copies of $P_3$ and $P_4$ used to define the trees of $\mathcal{F}$. Since $u_4$ is a leaf in $T'_3$, $u_4$ is the vertex in $C$ of a copy $P_3^0$ of $P_3$. Let $v$ be the vertex in $B$ and $w$ the vertex in $A$ of $P_3^0$. 
    Let $P_3^1$ be a copy of the path $P_3$ induced by $\{u_0,u_1,u_2\}$ and let $P_4^1$ be a the copy of the path $P_4$ induced by $\{u_3, u_4, v, w\}$. 
     The tree $T$ can be obtained from $P_3^1$, $P_4^1$ and the copies of $P_3$ and $P_4$ different from $P_3^0$ used to construct $T'_3$, so that the edges joining the copies have both endpoints in $A\cup X$ and no vertex in $A\cup X$ remains a leaf (see Figure~\ref{case1}).
     Hence, $T\in \mathcal{F}$.
    
\begin{figure}[h]
\centering
 \includegraphics[width=0.9\textwidth]{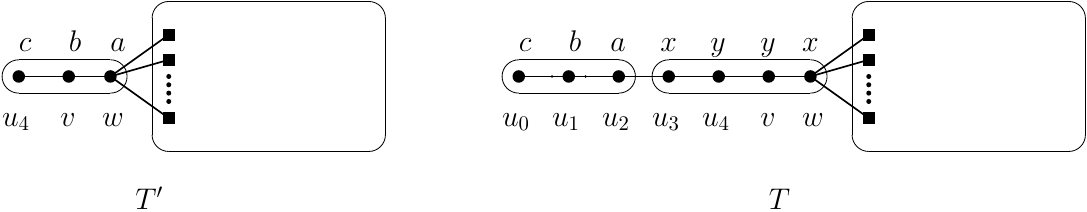}
\caption{The tree $T$ is in $\mathcal{F}$ if $T'\in \mathcal{F}$. Squared vertices belong to $A\cup X$.}
\label{case1}
\end{figure}

{\bf Case 2:} $\deg (u_3)\ge 3$.
Let $T'=T-\{ u_0,u_1,u_2\}$. Hence, $T'$ is a tree of order $n-3$ with $\ls-1$ leaves.
If $D'$ is a minimum isolating set of $T'$, then $D=D'\cup \{ u_1 \}$ is an isolating set of $T$. 
 Thus,
    
$$\up =\iota (T)\le |D|=|D'|+1=\iota (T')+1\le \frac {(n-3)+(\ls-1)}4+1=\up$$
implying that all inequalities in the preceding expression must be equalities. Hence, 

$$\iota (T')=\frac {(n-3)+(\ls-1)}4,$$    
and, by the induction hypothesis, $T'\in \mathcal{F}$.
Consider the construction of $T'$ from copies of $P_3$ and $P_4$ used to define the trees of $\mathcal{F}$. 
Notice that $u_3\notin C$, because $\deg_{T'}(u_3)\ge 2$.
Besides, $u_3\notin B$, by Lemma~\ref{pending_p5}.
Let $P_3'$ be the path of $T$ induced by $\{u_0,u_1,u_2\}$ and label the vertex $u_2$ with $a$.
Since $\deg (u_3)\ge 3$, from Remark~\ref{propertiesF} $e)$,
we know that $u_3\notin Y$.
If $u_3\in A\cup X$, then $T\in \mathcal {F}$, because $T$ can be obtained by adding to $T'$ the copy $P_3'$ and the edge $u_2u_3$ with endpoints in $A\cup X$.
\end{proof}

Although the bound given in Theorem~\ref{th:orderplusleaves} is sharp, the difference between this bound and the isolation number of a tree can be arbitrarily large.

\begin{Prop}\label{large}
For every integer $k\ge 1$, there is a tree $T$ such that $\frac{n+\ell}{4} - \iota(T) =k$, where $n$ and $\ell$ are the order and the leaf order of $T$, respectively.
\end{Prop}
\begin{proof} Let $k\geq 1$ be an integer.
Consider a tree $T$ obtained by attaching $2k-1$ leaves to a support vertex of a path of order $4$.
Then, $T$ is a tree of order $2k+3$ with $2k+2$ leaves and isolation number equal to 1. Therefore

$$ \frac{n+\ell}{4}-\iota(T)= \frac{(2k+3)+(2k+1)}{4}-1=k.$$
\end{proof}

We finish this section by proving an upper bound on the isolation number for trees in terms of order, leaves and support vertices.
The following bound is proved in \cite{boutrig}.

\begin{Prop}\cite{boutrig}\label{prop:boundboutrig}
If $T$ is a tree of order $n\ge 3$ and leaf order $\ls$, and $T$ has exactly $s$ support vertices, then $\iota(T)\le \frac{n-\ell+s}{3}$.
\end{Prop}

Using an argument similar to that in \cite{boutrig}, we can get that the former bound is also true for general graphs.
Using this result, the bound given in Theorem~\ref{th:orderplusleaves} can be improved if the number of support vertices is known.

\begin{Thm}\label{thm:boundleavessupport}
If $T$ is a tree of order $n\ge 3$ and leaf order $\ls$, and $T$ has exactly $s \neq 1$ support vertices, then

 $$\iota(T) \leq \frac{n-\ell+ 2s}{4}.$$ 
 
\end{Thm}
\begin{proof} 
Let $T^*$ be the tree obtained by deleting all but one of the leaves at each support vertex of $T$. Then, $T^*$ is a tree of order $n^*=n-\ell+s$ and leaf order $\ell^*=s$. Observe that $n^* \ge 3$. Since all graphs have minimum isolating sets without leaves, $\iota (T)=\iota (T^*)$. Hence, by Theorem~\ref{th:orderplusleaves},

$$\iota (T)=\iota (T^*)\le \frac{n^*+\ell^*}{4}=
\frac{(n-\ell+s)+s}{4}=
\frac{n-\ell+2s}{4}.$$
\end{proof}

Notice that the bound given in 
Theorem~\ref{thm:boundleavessupport} is better than the one given in Theorem~\ref{th:orderplusleaves}, since the number of support vertices is at most the number of leaves. Moreover, they are equal if and only if $T$ has no strong support vertex. Next, we characterize all trees achieving this last upper bound.
\medskip

\begin{Prop}
Let $T$ be a tree of order $n$ and leaf order $\ls$, and having exactly $s \ge 2$ support vertices. Then,
$\iota(T)=\frac{n-\ell+2s}4$ if and only if $T$ is obtained by adding twin leaves to a tree in $\mathcal{F}$ (see an example in Figure~\ref{fig:familytreesleavessupport}).
\end{Prop}
\begin{proof}
Let $T^*$ be the tree obtained by deleting all but one of the leaves at each support vertex of $T$. Then, $\iota (T)=\iota (T^*)$.

\begin{figure}[ht!]
\centering
 \includegraphics[width=0.45\textwidth]{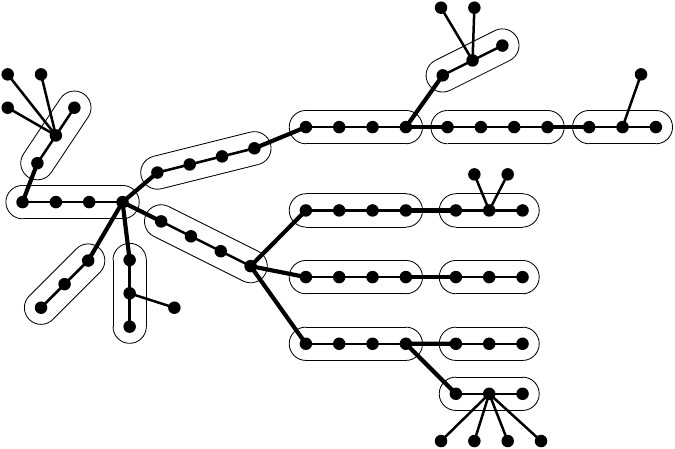}
\caption{A tree in $\mathcal{F}$ that attains the upper bound in Theorem~\ref{thm:boundleavessupport}.}
\label{fig:familytreesleavessupport}
\end{figure}

In addition, if $T$ has exactly $n$ vertices, $\ell$ leaves and $s$ support vertices, and $T^*$ has exactly $n^*$ vertices, $\ell^*$ leaves and $s^*$ support vertices,
then $n^*=n-\ell+s\geq 2+s\geq 4$ and $\ell^*=s^*=s \geq 2.$

Therefore, if $T^*\in \mathcal{F}$, then

$$\iota(T)=\iota (T^*)=\frac{n^*+\ell^*}4=\frac{(n-\ell+s)+s}{4}=\frac{n-\ell+2s}4.$$ Hence, $T$ attains the upper bound in Theorem~\ref{thm:boundleavessupport}.

Conversely, suppose that $T$ attains the upper bound in Theorem~\ref{thm:boundleavessupport}. Then,  $n^*= n-\ell+ 2s \geq 2+ 2s \geq 6 $  and

$$\iota(T^*)=\iota(T)=\frac{n-\ell+2s}4=\frac{n^*-s+2s}{4}=\frac{n^*+s}4=\frac{n^*+\ell^*}4.$$
Using Theorem \ref{charactupperbound}, we get  $T^*\in \mathcal{F}.$
Notice that $T$ is obtained by adding twin leaves to $T^*$, so we are done.
\end{proof}

\color{black}



\section{An upper bound on the $k$-isolation number of a tree}\label{sec:upperboundsstarisolation}

In this section, we establish a sharp upper bound on the $k$-isolation number of a tree $T$ in terms of $k$, the order $n$ and the number of leaves $\ell$, and we determine the extremal trees. Note that we have done this for $k = 1$ in the previous section. Surprisingly, the bound we present here, given by the following result, is best possible for any $k \geq 2$ but inferior to that in Theorem~\ref{th:orderplusleaves} for $k = 1$, meaning that the problem for $k \geq 2$ is somewhat different from that for $k = 1$.

\begin{Thm}\label{th:iotakupperbound}  If $k \ge 2$ and $T$ is a tree of order $n$ and leaf order $\ell$, then

$$\iota_k(T) \leq \frac{n + \ell}{2k+1}.$$
\end{Thm}

\begin{proof} 
The proof is by induction on $n$. Trivially, $\iota_k(T) = 0$ if $\Delta < k$. Suppose $\Delta \ge k$. Then, $n \geq \Delta + 1$. By (\ref{lformula}), $\ell \geq 2 + n_{\Delta}(\Delta-2) \geq 2 + 1(\Delta - 2) = \Delta$. Thus, $n + \ell \geq 2\Delta + 1 \geq 2k+1$.

If $n + \ell = 2k+1$, then $\Delta = k$, $n = k+1$ and $\ell = k$, so $T \simeq K_{1,k}$ and $\iota_k(T) = 1 = \frac{n + \ell}{2k+1}$.

Now suppose $n + \ell \ge 2k+2$. Let $d=\diam (T)$.
We choose a longest path $U = (u_0, u_1, \dots, u_d)$ of $T$ with the largest value of $\deg_T(u_1)$. If $d = 1$, then $\{u_0\}$ is an isolating set of $T$, so $\iota(T) = 1 < \frac{n + \ell}{2k+1}$. Suppose $d \geq 2$. For any $i \in [d-1]$, let $T_i$ and $T_i'$ be as in the proof of Theorem~\ref{th:orderplusleaves}. Recall that the distance (in $T$) of any vertex of $T_i$ from $u_i$ is at most $i$, and the distance (in $T$) of any vertex of $T_i'$ from $u_i$ is at most $d-i$. As in the proof of Theorem~\ref{th:orderplusleaves}, $u_0$ and any other vertex in $N_{T_1}(u_1)$ is a leaf of $T$. If $d \le 4$, then $\{u_2\}$ is an isolating set of $T$, so we have $\iota_k(T) \leq \iota(T) = 1 < \frac{n+\ell}{2k+1}$. Suppose $d \ge 5$. 

{\bf Case 1: $\deg_T(u_1) < k$.} Let $T' = T - u_0$. Then $T'$ is a tree of order $n' = n-1$ and leaf order $\ell' \leq \ell$. Let $D'$ be a minimum $k$-isolating set of $T'$. 
Then, $D'$ is a $k$-isolating set of $T$. By the induction hypothesis, we have

$$\iota_k(T) \le |D'| \le \frac{n' + \ell'}{2k+1} \le \frac{n - 1 + \ell}{2k+1} < \frac{n + \ell}{2k+1}.$$

{\bf Case 2: $\deg_T(u_1) = k$.} Since $d\ge 5$, $\deg_T(u_4)\ge 2$.

Suppose $\deg_T(u_4) \ge 3$. Let $T'= T_3'$, $n' = |V(T')|$ and $\ell' = \ell(T')$. Then, $n' \leq n - (k+2)$. Recall that $N_{T_1}(u_1) \subseteq L(T)$. Since $N_{T_1}(u_1) = N_T(u_1) \setminus \{u_2\}$, $\ell' \le \ell - (\deg_T(u_1) - 1) = \ell - (k-1)$. Let $D'$ be a minimum $k$-isolating set of $T'$. By the choice of $U$, $D' \cup \{u_3\}$ is a $k$-isolating set of $T$ (see Figure~\ref{fig:demostardk}, left). By the induction hypothesis, we have

$$\iota_k(T) \le |D'| + 1 \le \frac{n'+\ell'}{2k+1} +1
\le \frac{n-(k+2)+\ell-(k-1)}{2k+1}+1=\frac{n+\ell}{2k+1}.$$

\begin{figure}[h]
\centering
 \includegraphics[width=0.4\textwidth]{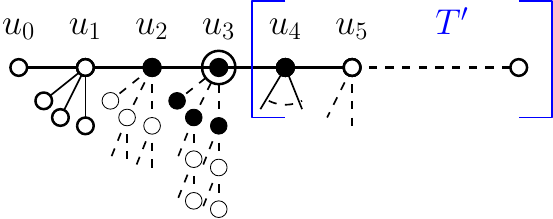}\hspace{8mm}
 \includegraphics[width=0.4\textwidth]{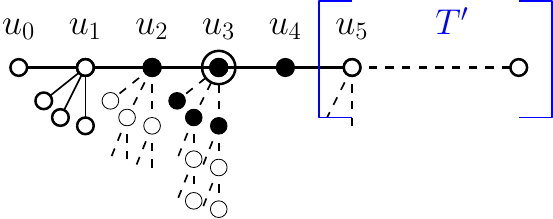}
\caption{In both cases, the set formed by $u_3$ and a $k$-isolating set of $T'$ is a $k$-isolating set of $T$.}
\label{fig:demostardk}
\end{figure}

Now suppose $\deg_T(u_4) = 2$. Let $T'= T_4'$, $n' = |V(T')|$ and $\ell' = \ell(T')$. Then, $n' \le n - (k+3)$ and $\ell' \le \ell - (k-2)$. Let $D'$ be a $k$-isolating set of $T'$.
By the choice of $U$, $D'\cup \{u_3\}$ is a $k$-isolating set of $T$ (see Figure~\ref{fig:demostardk}, right).
By the induction hypothesis, we have

$$\iota_k(T) \le |D'|+1\le \frac{n'+\ell'}{2k+1}+1
\le \frac{n-(k+3)+\ell-(k-2)}{2k+1}+1=\frac{n+\ell}{2k+1}.$$

%

{\bf Case 3: $\deg_T(u_1) > k$.} Suppose $\deg_T(u_3)\ge 3$.  Let $T'= T_2'$, $n' = |V(T')|$ and $\ell' = \ell(T')$. Then, $n' \le n-(k+2)$. Since we established that $N_T(u_1) \setminus \{u_2\} \subseteq L(T)$, $\ell' \le \ell - k$. Let $D'$ be a minimum $k$-isolating set of $T'$. Then, $D'\cup \{u_2 \}$ is a $k$-isolating set of $T$
(see Figure~\ref{fig:demostardmayork}, left).
By the induction hypothesis, we have

$$\iota_k(T)\le |D'| + 1 \le \frac{n'+\ell'}{2k+1} + 1 \le \frac{n-(k+2) + \ell - k}{2k+1} + 1 < \frac{n + \ell}{2k+1}.$$

\begin{figure}[h]
\centering
 \includegraphics[width=0.4\textwidth]{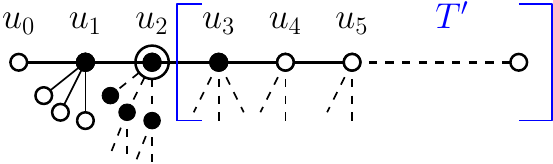}\hspace{8mm}
 \includegraphics[width=0.4\textwidth]{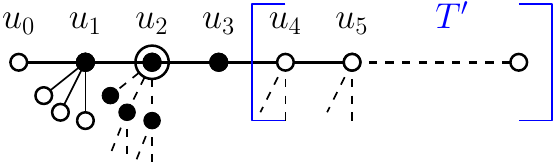}
\caption{In both cases, the set formed by $u_2$ and a $k$-isolating set of $T'$ is a $k$-isolating set of $T$.}
\label{fig:demostardmayork}
\end{figure}

Now suppose $\deg_T(u_3) = 2$. Let $T'= T_3'$, $n' = |V(T')|$ and $\ell' = \ell(T')$. Then, $n'\le n-(k+3)$ and $\ell'\le \ell - (k-1)$. Let $D'$ be a minimum $k$-isolating set of $T'$. Then, $D'\cup \{u_2 \}$ is a $k$-isolating set of $T$ (see Figure~\ref{fig:demostardmayork}, right). By the induction hypothesis, we have

$$\iota_k(T) \le |D'| + 1 \le \frac{n' + \ell'}{2k+1} + 1
\le \frac{n - (k+3) + \ell - (k-1)}{2k+1} + 1 < \frac{n+\ell}{2k+1}.$$

%
\end{proof}

We now have all the results needed for proving Theorem~\ref{thm:generaltreebound}.

\begin{proof}[Proof of Theorem~\ref{thm:generaltreebound}] Part (a) follows immediately from Theorems~\ref{thm:orderminusleaves_k} and \ref{th:orderplusleaves}. Part (b) follows immediately from Theorems~\ref{thm:orderminusleaves_k}, \ref{th:iotakupperbound} and \ref{thm:caro_trees}.
\end{proof}

The rest of the section is dedicated to characterizing the trees attaining the upper bound in Theorem~\ref{th:iotakupperbound}, initially treating the case $\iota_k(T) = 1$.

For any tree $T$, the number of vertices of $T$ of degree at least $k \geq 2$ will be denoted by $n_{\ge k}(T)$, or simply by $n_{\ge k}$ where no confusion arises.

\begin{Rem} \label{b1} Let $T$ be a tree with $\Delta(T) \geq 1$. Let $\ell = \ell(T)$ and $\Delta = \Delta(T)$. By (\ref{lformula}), for $2 \leq k \leq \Delta$,

$$\ell =2 + \sum_{i=3}^{\Delta} n_i (i-2) \ge 2+n_{\ge k} (k-2).$$
Since $ n+\ell \geq   (n_{\ge k}+\ell ) + \ell \geq 1 + 2 \ell \ge 1+2k$,

$$n+\ell\geq {2k+1}.$$
\end{Rem}
 
The following result establishes a sufficient condition on the order of a tree $T$ so that $\iota_k(T) = 1$. 

\begin{Cor}\label{2k+1}  Let $T$ be a tree of order $n \geq 3$ and $2\leq k \leq  \Delta (T)$. If $n \leq 2k+1$, then $\iota_k(T) = 1$. Moreover, $n + \ell = 2k+1$ if and only if $T \simeq K_{1,k}$.
\end{Cor}

\begin{proof} Since $2 \leq k \leq \Delta (T)$, $n_{\geq k} \geq 1$, so $\iota_k(T) \geq 1$. Assume $n \leq 2k+1$.

If $n_{\ge k} = 1$, then the set consisting of the unique vertex of degree at least $k$ is a $k$-isolating set, so $\iota_k(T) = 1$.

Suppose $n_{\ge k} = 2$. Let $u$ and $v$ be the vertices of degree at least $k$. If ${\rm dist}_T(u,v) \ge 3$, then $n \ge |N[u]| + |N[v]| \geq 2(k+1)$, which contradicts $n \leq 2k+1$. Thus, ${\rm dist}_T(u,v) \leq 2$, and hence $N[u] \cap N[v] \neq \emptyset$. Let $w \in N[u] \cap N[v]$. Then, $\{w\}$ is a $k$-isolating set of $T$, so $\iota_k(T) = 1$.

Now, suppose $n_{\ge k} \ge 3$. By Remark~\ref{b1}, $\ell \geq 2 + 3(k-2)$. We have $2k+1 \ge n \ge 3 + \ell \ge 3k-1$, so $k=2$ and $n=5$. Thus, $T \simeq P_5$, and hence $\iota_k(T) = 1$.

Finally, suppose $n + \ell = 2k+1$. Since $n + \ell \ge 1 + 2 \ell \ge 2k+1$, we have $\ell = k$ and $n = k+1$, so $T \simeq K_{1,k}$. Conversely, if $T \simeq K_{1,k}$, then $n + \ell = 1 + 2k$.
\end{proof}

\begin{ex}  We provide  some  examples   of   trees  with $\iota_k(T) \leq 1$. 

\begin{enumerate}
\item
 $\iota_k(T)=0$ if and only if   $n_{\ge k}(T)=0$ if and only if
$\Delta (T)<k.$ 
 \item
 If $T \simeq P_{n}$ with $n \geq 3,$ then $\Delta(T) = 2$, $n_2 = n - 2$ and $\ell = n_1 = 2$. For $3 \leq n \leq 7$, $\iota_2(P_{n}) = 1$. 
 \item
If $T \simeq K_{1,k}$ with $k\geq 2$, then $\Delta(T) = k$, $n_{\ge k}(T) = 1$ and $\iota_k(T) = 1$.
\item If $n_{\ge k}(T)=1,$ then the set consisting of the unique vertex of degree at least $k$ is a $k$-isolating set of $T$, so $\iota_k(T) = 1.$
  \item  
  If $diam (T)\leq 4$, then the set consisting of a central vertex is an isolating set of $T$, so $\iota_k(T) = 1$.
\end{enumerate}
\end{ex}

\begin{Cor}\label{th:sharpsmallorder} Let $T$  be a tree of order $n \geq 3$ and $2\leq k \leq \Delta (T)$. Then, $n \le 2k+3$ and $\iota_k(T)=\frac{n+\ell}{2k + 1}$ if and only if $T \simeq K_{1,k}$.   
\end{Cor}

\begin{proof} If $T \simeq K_{1,k}$, then $\iota_k(T) = 1 = \frac{n+\ell}{2k+1}$ and $n = k+1 < 2k+3$. We now prove the converse. Thus, suppose $n \le 2k+3$ and $\iota_k(T) = \frac{n+\ell}{2k + 1}$. Let $m = |E(T)|$. Since $T$ is a tree, $m = n - 1 \leq 2k + 2$. Let $v \in V(T)$ with $\deg_T(v) = \Delta$. Let $T' = T - N[v]$ and $n' = |V(T')|$. Then, $n' = n - \Delta - 1 \leq 2k + 3 - \Delta - 1$. 

Suppose $\iota_k(T) \geq 2$. Then, $T'$ contains a $k$-star $S$. Thus, $n' \geq |V(S)| = k+1$. 

Suppose $xy \notin E(T)$ for every $x \in N(v)$ and $y \in V(S)$. Since $T$ is connected and $2k + 2 \leq \Delta + 1 + |V(S)| = |N[v]| + |V(S)| \leq n \leq 2k + 3$, it clearly follows that $\deg_T(v) = k$ and $T$ has a vertex $w \notin N[v] \cup V(S)$ such that $V(T) = N[v] \cup V(S) \cup \{w\}$ and $x'w, wy' \in E(T)$ for some $x' \in N(v)$ and $y' \in V(S)$. Since $m \leq 2k+2$, $E(T) = \{vx \colon x \in N(v)\} \cup E(S) \cup \{x'w, wy'\}$. Thus, $T - N[w]$ contains no $k$-star, but this contradicts $\iota_k(T) \geq 2$. 

Therefore, $xy \in E(T)$ for some $x \in N(v)$ and $y \in V(S)$. Let $y^*$ be the vertex of $S$ with $\deg_S(y^*) = k$. Let $T'' = T - N[x]$. Since $v, y \in N[x]$, $|E(T'')| \leq m - (\deg_T(v) + \deg_S(y)) - 1$. Suppose $\deg_T(v) \geq k+1$ or $\deg_S(y) \geq 2$. Since $\deg_T(v) = \Delta \geq k$ and $\deg_S(y) \geq 1$, $|E(T'')| \leq (2k + 2) - (k + 2) - 1 = k - 1$, so $T''$ contains no $k$-star, but this contradicts $\iota_k(T) \geq 2$. Thus, $m = 2k + 2$, $\deg_T(v) = k$, $\deg_S(y) = 1$, $|E(T'')| = (2k + 2) - (k + 1) - 1 = k$, 
and hence $T''$ is a $k$-star. Since $k \geq 2$, $y^* \neq y$. Thus, the $k$ edges of $T''$ are the $k-1$ edges of $S-y$, which are all incident to $y^*$, and $wy' \notin E(S-y)$ for some $w \in V(T')$ and $y' \in V(S-y) = N_{S-y}[y^*]$. Since $\deg_T(v) = \Delta = k$, $\deg_T(y^*) = k$. Thus, $y' \neq y^*$, and hence $y' \in N(y^*)$. Since $wy' \notin E(S-y)$ and $T$ contains no cycles, $w \notin V(S-y)$. Since $T''$ is a $k$-star, we conclude that $k = 2$. Let $x'$ be the neighbour of $v$ that is not $x$. The $m = 2k + 2 = 6$ edges of $T$ are $x'v, vx, xy, yy^*, y^*y', y'w$ (so $T \simeq P_7$). Thus, $\{y\}$ is a $k$-isolating set of $T$, but this contradicts $\iota_k(T) \geq 2$.

Therefore, $\iota_k(T) = 1$. Together with $\iota_k(T) = \frac{n + \ell}{2k + 1}$, this gives us $n + \ell = 2k+1$, so $T \simeq K_{1,k}$ by Corollary \ref{2k+1}. 
\end{proof}

The following remark can be derived from the proof of Theorem~\ref{th:iotakupperbound}.  
 
\begin{Rem}\label{rem:diam5}
If a tree $T$ attains the bound in Theorem~\ref{th:iotakupperbound}, then $T \simeq K_{1,k}$ or $\diam(T) \geq 5$. Moreover, if $(u_0,u_1,\dots ,u_d)$ is a diametral path of $T$ maximizing the degree of $u_1$, then $\deg(u_1) = k$. 
\end{Rem}

A component of a graph $G$ is said to be \emph{non-trivial} if it has more than one vertex.

\begin{cons}\label{notation:familyT_k}
For every $k\ge 2$, let $\mathcal{T}_k$ be the family of trees constructed as follows.
The set of  vertices of $T\in \mathcal{T}_k$ is $V(T)=A(T)\cup B(T)\cup C(T)\cup L(T)$, where 
$A(T)$, $B(T)$, $ C(T)$ and $L(T)$ are pairwise disjoint, and
\begin{enumerate}[$\bullet$]
    \item $L(T)$ is the set of leaves of $T$ and all their support vertices belong to $C(T)$;
    \item $A(T)$ induces a forest whose $h$ components are all non-trivial;
    \item the subgraph of $T$ induced by $A(T)\cup B(T)$ is obtained by attaching exactly one leaf to each vertex of $A(T)$;
    \item for each $u\in B(T)$, $\deg_T(u)=2$ and $u$ has exactly one neighbour in $A(T)$ and another in $C(T)$;
    \item for each $u\in C(T)$, $\deg_T(u)=k$, $u$ has at least one neighbour in $B(T)$, and each neighbour of $u$ is either in $B(T)$ or in $L(T)$; 
    \item $A(T)\cup B(T)\cup C(T)$ induces a tree.
\end{enumerate}
Figure~\ref{fig:treeTk} illustrates a graph in $\mathcal{T}_k$ with $k = 4$.
\end{cons}
\begin{figure}[h]
\centering \includegraphics[width=0.7\textwidth]{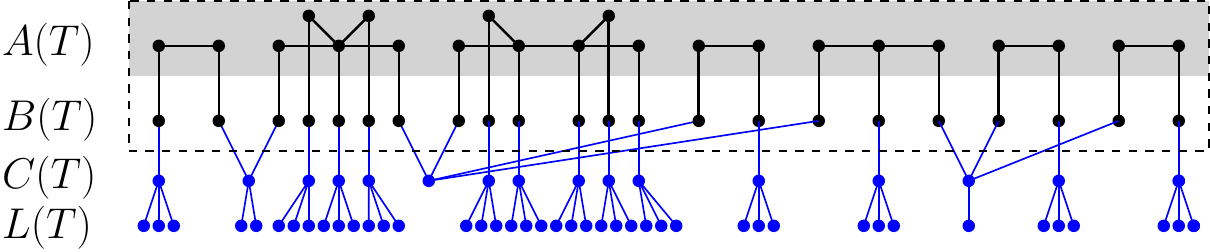}
\caption{A tree in $\mathcal{T}_4$. $A(T)$ induces a forest with non-trivial components. $A(T)\cup B(T)$ induces a corona graph. The vertices in $B(T)$ have degree 2 and all the vertices in $C(T)$ have degree 4.}
\label{fig:treeTk}
\end{figure} 
\vspace{3mm}

\begin{Lem}\label{lem:propsTk}
Let $T\in \mathcal{T}_k$, let $n_0$ be the order of the forest $A(T)$, and let $h$ be the number of components of $A(T)$. Then, $T$ has the following properties:
\begin{enumerate}[a)]
\item\label{itema} $\vert A(T)\vert=\vert B(T)\vert=n_0\ge 2h\ge 2$.
\item\label{itemb} $\vert C(T)\vert
=n_0-(h-1)$.
\item\label{itemc} $\vert L(T)\vert=
(k-1)n_0-k(h-1)$.
\item\label{itemd} $\vert V(T)\vert =
(k+2)n_0 -(k+1)(h -1)\ge 2k+4$.
\item\label{itemnou} $\vert V(T)\vert +
\vert L(T)\vert=(2k+1)|C(T)|$.
\item\label{itemnounou} Each vertex in $A(T)$ has exactly one vertex of $C(T)$ at distance $2$. Moreover, if $u,v\in A(T)$, then $d_T(u,v)=4$  if and only if the unique vertex in $C(T)$ at distance $2$ from $u$ is the unique vertex in $C(T)$ at distance $2$ from $v$. 
\item\label{itemh} $d(u,v)\ge 5$ for every $u,v\in C(T)$ with $u \neq v$.
\end{enumerate}
\end{Lem}
\begin{proof}
 Items~\ref{itema}),
 \ref{itemnounou}) 
 and \ref{itemh}) 
are direct consequences of the definition of $\mathcal{T}_k$.
Item \ref{itemb}) is derived from the known relation between order and size of a forest. 
Indeed, on the one hand, if $F_1$ is the forest induced by $A(T)\cup B(T)$, then 

$$\vert E(F_1)\vert =|V(F_1)|-h=\vert A(T)\cup B(T)\vert -h=2n_0-h.$$ 
On the other hand, if $T_0$ is the tree induced by
$A(T)\cup B(T)\cup C(T)$, then

$$\vert E(T_0)\vert =|V(T_0)|-1=\vert A(T)\cup B(T)\cup C(T)\vert -1=
2n_0+\vert C(T)\vert -1.$$ Besides, since there are exactly $n_0$ edges with one vertex in $B(T)$ and the other in $C(T)$, we have 
$\vert E(T_0)\vert =\vert E(F_1)\vert + n_0.$ Hence,

\begin{align*}
|C(T)|&=|E(T_0)|-2 n_0 +1
=(\vert E(F_1)\vert + n_0)-2 n_0 +1\\
&=\vert E(F_1)\vert -n_0 +1
=(2n_0-h)-n_0+1=n_0-(h-1).    
\end{align*}

Now let us prove item~\ref{itemc}). By construction and by item~\ref{itemb}), 

$$|L(T)|=k \, |C(T)|-n_0=k\,  (n_0-h+1)-n_0=(k-1)n_0-k(h-1).$$
By items \ref{itema}), \ref{itemb}) and \ref{itemc}),

\begin{align*}
|V(T)|&=|A(T)\cup B(T)\cup C(T)\cup L(T)|\\&=2n_0+(n_0-(h-1))+((k-1)n_0-k(h-1))\\&=  
(k+1) (n_0 -h +1)+n_0
\ge
(k+1) (h+1)+2 \\
&\ge
(k+1) 2+2 \ge 2k+4.
\end{align*}
Hence, item~\ref{itemd}) holds.
Item~\ref{itemnou}), is a direct consequence of  items~\ref{itemb}), \ref{itemc}) and \ref{itemd}).
\end{proof}

\begin{Prop}\label{prop:extremal1}
If $T\in \mathcal{T}_k$, $k\ge 2$, and $T$ is of order $n$ and leaf order $\ell$, then $n\ge 2k+4$ and

$$\iota_k(T)=\frac{n+\ell}{2\, k+1}.$$    
\end{Prop}
\begin{proof}
On one hand, $n\ge 2k+4$ by Lemma~\ref{lem:propsTk}. On the other hand, since the $k$-stars with centers in $C(T)$ are disjoint and the distance between any two such centers is at least $5$ (by item g) of Lemma~\ref{lem:propsTk}), every $k$-isolating set of $T$ must contain at least $\vert C(T)\vert$ vertices. 
Therefore, $\iota_k(T)\ge |C(T)|$. By Lemma~\ref{lem:propsTk}, $|C(T)|=\frac {n+\ell}{2k+1}$.
Hence, by Theorem~\ref{th:iotakupperbound}, $\iota_k(T)=\frac {n+\ell}{2k+1}$
\end{proof}

Next, we describe how to construct a minimum $k$-isolating set for trees of the family $\mathcal{T}_k$ (see an example in Figure~\ref{fig:iotakset}).

\noindent
\begin{cons}\label{construction}
Let $T\in \mathcal{T}_k$.
Consider the partition 
 $\{A_1,\dots ,A_h\}$ of $A(T)$ such that  $T[A_1], T[A_2]$,$ \dots$, $T[A_h]$ are the components of the forest induced by $A(T)$, and $V(T[A_i]) = A_i$ for each $i \in [h]$.
By definition of $\mathcal{T}_k$, for each $i\in [h]$, there exist vertices $v_i\in A_i$ and $v\in A(T)\setminus A_i$ such that $d(v_i,v)=4$. 
By Lemma~\ref{lem:propsTk} f), the unique vertex in $C(T)$ at distance $2$ from $v_i$ is the unique vertex in $C(T)$ at distance $2$ from $v$.
\vspace{2mm}

We recursively construct a  set $S=\{u_{i_2},\dots, u_{i_h}\}\subseteq V(T)$, such that 
 $u_{i_j}\in A_{i_j}$ for $2\le j\le h$, and
$\{i_2,\dots ,i_h\}=\{2,\dots ,h\}$ in  the following way.

At the first step, we begin by choosing 
a vertex $u_{i_2}\in A_{i_2}$, $A_{i_2}\not= A_1$, such that $u_{i_2}$ is at distance 4 from a vertex $u_1\in A_1$.
This is always possible since $T$ is connected. 

At step $r$, $2\le r<h-1$, suppose we have already chosen vertices $\{u_{i_2},\dots ,u_{i_r}\}$
such that
$u_{i_j}\in A_{i_j}$, for $2\le j\le r$, where 
$A_1,A_{i_2},\dots,A_{i_r}$ are different sets, and $u_{i_j}$ has a vertex at distance 4 in $(A_1\cup A_{i_2}\cup \dots \cup A_{i_{j-1}})\setminus \{u_{i_2},\dots ,u_{i_{j-1}}\}$.
Then, choose a vertex $u_{i_{r+1}}\in A_{i_{r+1}}$ such that
the sets 
$A_{1}, A_{i_2}, \dots ,A_{i_r}, A_{i_{r+1}}$ are  different and $u_{i_{r+1}}$ has a vertex at distance 4 in $(A_{1}\cup A_{i_2}\cup \dots \cup A_{i_r})\setminus \{u_{i_2},\dots, u_{i_r}\}$. Such a  vertex $u_{i_{r+1}}$ exists. Indeed,  since $T$ is connected, there must be a vertex at distance 4 from some vertex in $A_{1}\cup A_{i_2}\cup \dots \cup A_{i_r}$. Moreover, if $u_{i_{r+1}}$ is at distance 4 from all the vertices $u_{i_2},\dots, u_{i_r}$, then, by Lemma~\ref{lem:propsTk}~\ref{itemnounou}), it must be also at distance 4 from $u_1\in A_1$.

Let $D=A(T)\setminus \{u_{i_2},\dots ,u_{i_h}\}$.
\end{cons}
\begin{Rem}\label{rem:setD}
The set $D$ is obtained by  removing exactly one vertex from all but one of the $h$ components of the forest induced by $A(T)$ in such a way that every vertex in $C(T)$ has at least one vertex of $D$ at distance $2$.  
\end{Rem}
\begin{figure}[h]
\centering \includegraphics[width=0.9\textwidth]{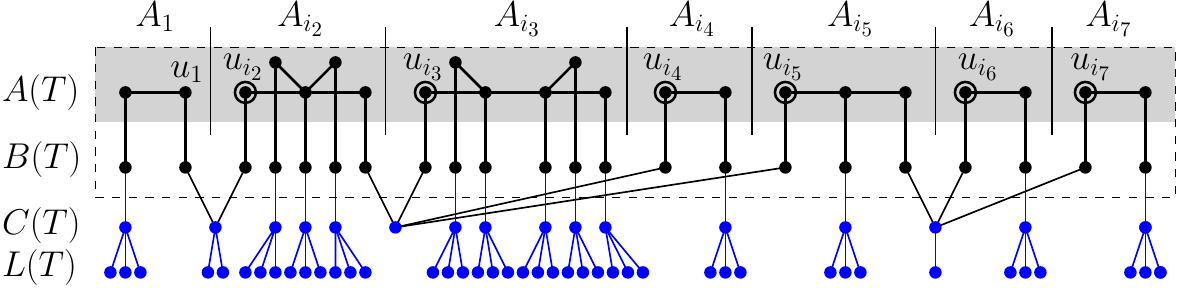}
\caption{A tree $T$ of $\mathcal{T}_4$. The set $D$ of non-circled vertices of $A(T)$ is $4$-isolating.}
\label{fig:iotakset}
\end{figure} 

\begin{Prop}\label{prop:isolatingset} 
If $T\in \mathcal{T}_k$ with $k\ge 2$, then any set $D$ obtained using Construction~\ref{construction} provides a minimum $k$-isolating set of $T$.
\end{Prop}
\begin{proof}
Let $D$ be a set obtained using Construction~\ref{construction}. 
Notice that $|D|=n_0-(h-1)$ and, by Lemma~\ref{lem:propsTk} and Proposition~\ref{prop:extremal1}, $n_0-(h-1)=|C(T)|=\frac{n+\ell}{2k+1}=\iota_k(T)$. Hence, it only remains to prove that $D$ 
is $k$-isolating.
To this purpose, we check that $T-N[D]$ has maximum degree at most $k-1$.
Notice that $N[D]\subseteq A(T)\cup B(T)$ and $A(T)\subseteq N[D]$.
Hence,  $T-N[D]=T[W]$, where $W=B'(T)\cup C(T)\cup L(T)$ for some subset $B'(T)\subseteq B(T)$. The vertices of  $B'(T)\cup L(T)$ are leaves in $T[W]$. Now, let $w\in C(T)$. By Remark~\ref{rem:setD}, every vertex  $w\in C(T)$ is at distance 2 from some vertex of $D$. 
Therefore, $w$ has degree at most $k-1$ in $T-N[D]$. Hence, $D$ is $k$-isolating. 
\end{proof}

\begin{Thm}
\label{thm:extremal2}
If $T$ is a tree of order $n$ and leaf order $\ell$, $k\ge 2$, $n\ge 2k+4$ and

$$\iota_k(T)=\frac{n+\ell}{2\, k+1},$$
then $T\in \mathcal{T}_k$.
\end{Thm}
\begin{proof} 
The proof is by induction on $n$. Let $T$ be a tree attaining the upper bound in Theorem~\ref{th:iotakupperbound}.

Suppose $n=2k+4$. Then, $ (2k+1)\iota_k(T)=n+\ell\ge 2k+6$, so $\iota_k(T)\ge 2$, and hence $T$ is not a star. 
Moreover, if  $\iota_k(T)\ge 3$, then 
$\iota_k (T)(2k+1)\ge 3(2k+1)=6k+3$, but
$\iota_k (T)(2k+1)=n+\ell \le (2k+4)+(2k+2) =4k+6$, a contradiction. Therefore, $\iota_k(T)= 2$ and $n+\ell=4k+2$, implying that $\ell=4k+2-n=2k-2$.
If $k=2$, then $\ell=2$. Hence, $T$ is a path of order $8$, that belongs to $\mathcal{T}_2$, for $h=1$ and $|A(T)|=|B(T)|=2$.
Suppose now $k\ge 3$.
If there are at least $3$ vertices of degree at least $k$, then $\ell\ge 2+3(k-2)=3k-4$, a contradiction since $\ell=2k-2$ and $k\ge 3$. Thus, there are exactly $2$ vertices, say $u$ and $v$, of degree at least $k$. Hence, $d(u,v)\ge 3$, because otherwise a common neighbour of $u$ and $v$ would be a $k$-isolating set, and this contradicts $\iota_k(T)=2$. If the degree of $u$ or $v$ is greater than $k$, then $\ell\ge 2+(k-2)+(k-1)=2k-1$, which contradicts $\ell=2k-2$.
Thus, $\deg(u)=\deg(v)=k$. 
If $d(u,v)\le 4$, then the set formed by a vertex at distance at most $2$ from $u$ and from $v$ would be a $k$-isolating set, implying $\iota_k(T)=1$, a contradiction. If $d(u,v)\ge 6$, then $n\ge 7+\ell\ge 7+2k-2=2k+5>2k+4$, a contradiction. Thus, $d(u,v)=5$ and $\deg_T(u)=\deg_T(v)=k$. Therefore, $T$ consists of the vertices $u$ and $v$, the $4$ vertices of the path from $u$ to $v$, exactly $k-1$ leaves attached to $u$, and exactly $k-1$ leaves attached to $v$, because otherwise $n>6+2(k-1)=2k+4$, a contradiction. Notice that the described tree belongs to $\mathcal{T}_k$, for $h=1$ and $|A(T)|=|C(T)|=2$.
\vspace{1mm}

Now suppose $n > 2k+4$. Let $d=\diam (T)$. By the proof of Theorem~\ref{th:iotakupperbound}, $T \simeq K_{1,k}$ or $d\ge 5$. 
Since $n > 2k+4$, $d\ge 5$.

As in the proof of Theorem~\ref{th:iotakupperbound}, 
among all diametral paths of $T$ choose one path $P=(u_0,u_1,\dots ,u_d)$ maximizing the degree of $u_1$. As argued in Remark~\ref{rem:diam5}, $\deg_T(u_1)=k$.

\begin{claim}
$\deg_T(u_2)=2$.    
\end{claim}
\begin{proof}
Suppose to the contrary $\deg_T(u_2)\ge 3$.
We distinguish two cases.

If $\deg_T(u_3)\ge 3$, then let 
 $T'$ be the component containing $u_3$ after the removal of the edge $u_2u_3$ from $T$. Thus, $T'$ is a tree of order $n'$ and $\ell'$ leaves, with $n'\le n-(k+2)$ and $\ell'\le \ell-k$. If $D'$ is a minimum $k$-isolating set of $T'$, 
 then $D'\cup \{ u_2 \}$ is a $k$-isolating set of $T$.
 Hence, by Theorem~\ref{th:iotakupperbound},

{\small $$\iota_k(T)\le 1+\iota_k(T')\le 1+\frac{n'+\ell'}{2k+1}\le 
 1+\frac{n-(k+2)+\ell-k}{2k+1}=  \frac{n+\ell-1}{2k+1} 
 <\frac{n+\ell}{2k+1} $$}
 which is a contradiction.

If $\deg_T(u_3)= 2$, then let 
 $T'$ be the component of $T-u_3u_4$ containing $u_4$. Thus, $T'$ is a tree of order $n'$ and $\ell'$ leaves, 
 with $n'\le n-(k+3)$ and $\ell'\le \ell-(k-1)$. If $D'$ is a minimum $k$-isolating set of $T'$, 
 then $D'\cup \{ u_2 \}$ is a $k$-isolating set of $T$.
 Hence, by Theorem~\ref{th:iotakupperbound},

 {\small\begin{align*}
  \iota_k(T)&\le 1+\iota_k(T')\le 1+\frac{n'+\ell'}{2k+1}\le 
 1+\frac{n-(k+3)+\ell-(k-1)}{2k+1}=  \frac{n+\ell-1}{2k+1} \\
 &<\frac{n+\ell}{2k+1}    
 \end{align*}}
 which is again a contradiction.
\end{proof}

\begin{claim}
$\deg_T(u_3)=2$.   
\end{claim}
\begin{proof}
Suppose to the contrary $\deg_T(u_3)\ge 3$.
Let 
 $T'$ be the component containing $u_4$ in $T-u_3u_4$. Thus, $T'$ is a tree of order $n'$ and $\ell'$ leaves, 
 with $n'\le n-(k+3)$ and $\ell'\le \ell-(k-1)$. If $D'$ is a minimum $k$-isolating set of $T'$, 
 then $D'\cup \{ u_3 \}$ is a $k$-isolating set of $T$.
 Hence, by Theorem~\ref{th:iotakupperbound},

{\small \begin{align*}
  \iota_k(T)&\le 1+\iota_k(T')\le 1+\frac{n'+\ell'}{2k+1}\le 
 1+\frac{n-(k+3)+\ell-(k-1)}{2k+1}=  \frac{n+\ell-1}{2k+1} \\
 &<\frac{n+\ell}{2k+1}    
 \end{align*}}
 which is a contradiction.
\end{proof}

\begin{claim} 
$\deg_T(u_5)\ge 2$.   
\end{claim}
\begin{proof}
Suppose to the contrary  $\deg_T(u_5)=1$. Then, $\{u_3\}$ is a $k$-isolating set of $T$.
Thus, if $\deg_T(u_4)=r$, we have

$$\frac{n+\ell}{2k+1}=\frac{(r+k+2)+(r+k-2)}{2k+1}=\frac{2k+2r}{2k+1}>1=\iota_k(T),$$
which contradicts that $T$ attains the upper bound.
\end{proof}

Therefore, $\deg_T(u_2)=\deg_T(u_3)=2$ and $\deg_T(u_5)\ge 2$. Now we distinguish cases according to the degrees of $u_4$ and $u_5$. 

\begin{enumerate}[]

    \item {\bf Case 1:} $\deg_T(u_4)=2$ and $\deg_T(u_5)\ge 3$. 
    Let 
 $T'$ be the component containing $u_5$ of $T-u_4u_5$. Thus, $T'$ is a tree of order $n'$ and $\ell'$ leaves, 
 with $n'\le n-(k+3)$ and $\ell'\le \ell-(k-1)$. If $D'$ is a minimum $k$-isolating set of $T'$, 
 then $D'\cup \{ u_3 \}$ is a $k$-isolating set of $T$.
 Hence, by Theorem~\ref{th:iotakupperbound},
{\small \begin{align*}
  \iota_k(T)&\le 1+\iota_k(T')\le 1+\frac{n'+\ell'}{2k+1}\le 
 1+\frac{n-(k+3)+\ell-(k-1)}{2k+1}=  \frac{n+\ell-1}{2k+1} \\
 &<\frac{n+\ell}{2k+1}    
 \end{align*}}
 which is a contradiction.

\item 
{\bf Case 2:} $\deg_T(u_4)=\deg_T(u_5)=2$.
In such a case, $\deg_T(u_6)\ge 2$, since otherwise $n=k+5<2k+4$, a contradiction.

\begin{enumerate}
    \item {Case 2.1:  $\deg_T(u_6)=2$.} If $k\ge 3$, let $T'$ be the component of $T-u_5u_6$ containing $u_6$. Then, $T'$ is a tree of order $n'$ with $\ell'$ leaves, where $n'=n-(k+4)$ and $\ell'=\ell -(k-2)$. If $D'$ is a minimum $k$-isolating set of $T'$ and $k\ge 3$, then $D'\cup \{ u_3 \}$ is a $k$-isolating set of $T$. Hence,

     \begin{align*}
  \iota_k(T)&\le 1+\iota_k(T')\le 1+\frac{n'+\ell'}{2k+1}= 
 1+\frac{n-(k+4)+\ell-(k-2)}{2k+1}\\&=  \frac{n+\ell-1}{2k+1}
 <\frac{n+\ell}{2k+1}    
 \end{align*}
 which is a contradiction.

If $k=2$, then 
 let $T'$ be the component of $T-u_4u_5$ containing $u_5$. If $D'$ is a minimum $K_{1,2}$-isolating set of $T'$, then $D'\cup\{u_3\}$ is a $K_{1,2}$-isolating set of $T$. Hence,

{\small\begin{align*}
 \frac{n+\ell}{5} &= \iota_2(T)\le 1+\iota_2(T')\le 1+\frac{n'+\ell'}{5}= 
 1+\frac{n-5+\ell}{5}=\frac{n+\ell}{5}  
 \end{align*}}
implying that $T'$ attains the upper bound. Thus, $T'$ is either the star $K_{1,2}$ or belongs to $\mathcal{T}_2$, by the induction hypothesis.
If $T'$ is a star $K_{1,2}$, then $u_5$ must be a leaf of $K_{1,2}$, since $\deg_T(u_5)=2$. In such a case, $T$ is a path of order 8, that belongs to $\mathcal{T}_2$.
If $T'\in \mathcal{T}_2$, then $u_5$ must be a leaf of $T'$. Hence, $T\in \mathcal{T}_2$, by taking $A(T)=A(T')\cup \{u_3,u_4\}$, $B(T)=B(T')\cup \{u_2,u_5\}$ and $C(T)=C(T')\cup \{u_1\}$.

    \item {Case 2.2: $\deg_T(u_6)\ge 3$.} Let $T'$ be the component containing $u_5$ after the removal of the edge $u_4u_5$ from $T$. Then, $T'$ is a tree of order $n'$ with $\ell'$ leaves, where $n'=n-(k+3)$ and $\ell'=\ell-(k-2)$. If $D'$ is a minimum $k$-isolating set of $T'$, then $D'\cup \{ u_3 \}$ is a $k$-isolating set of $T$. Hence,

    \begin{align*}
  \frac{n+\ell}{2k+1}&=\iota_k(T)\le 1+\iota_k(T')\le 1+\frac{n'+\ell'}{2k+1}\\&= 
 1+\frac{n-(k+3)+\ell-(k-2)}{2k+1}=  \frac{n+\ell}{2k+1}.  
 \end{align*}
  Hence, all the inequalities in the preceding expression are equalities. Therefore, $T'$ attains the upper bound, implying that $T'$ is a $k$-star  or $T'\in \mathcal{T}_k$, 
  if $n'\ge 2k+4$, by the induction hypothesis.
  In both cases, $u_5$ is a leaf of $T'$.
  If $T'$ is a $k$-star, then
  $T\in \mathcal{T}_k$, 
  by taking $A(T)=\{u_3,u_4\}$, $B(T)=\{u_2,u_5\}$ and $C(T)=\{u_1,u_6\}$.
  If $T'\in \mathcal{T}_k$, then $u_6\in C(T')$,  since $u_5$ is a leaf in $T'$. Hence, we can check that $T\in \mathcal{T}_k$ by considering $B(T)=B(T')\cup \{x,t\}$; $A(T)=A(T')\cup \{z,w\}$; $C(T)=C(T')\cup \{y\}$. Therefore, $T\in \mathcal{T}_k$.
\end{enumerate}

\item 
{\bf Case 3:} $\deg_T(u_4)\ge 3$.
Let $T'$ be the component containing $u_4$ after the removal of the edge $u_3u_4$ from $T$. Then, $T'$ is a tree of order $n'$ with $\ell'$ leaves, where $n'=n-(k+2)$ and $\ell'=\ell-(k-1)$. If $D'$ is a minimum $k$-isolating set of $T'$, then $D'\cup \{ u_3 \}$ is a $k$-isolating set of $T$. Hence,

    \begin{align*}
  \frac{n+\ell}{2k+1}&=\iota_k(T)\le 1+\iota_k(T')\le 1+\frac{n'+\ell'}{2k+1}\\&= 
 1+\frac{n-(k+2)+\ell-(k-1)}{2k+1}=  \frac{n+\ell}{2k+1}.  
 \end{align*}

Thus, all the inequalities of the preceding expression must be equalities, implying that $T'$ achieves the upper bound. Then either $T'$ is a $k$-star or, by the induction hypothesis, $T'\in \mathcal{T}_k$.

If $T'$ is a $k$-star,
then $n=4+(k+1)<2k+4$, a contradiction. 
Therefore, $T'\in \mathcal{T}_k$. Then, $u_4\in A(T')\cup B(T')\cup C(T')$, since $\deg_{T'}(u_4)\ge 2$.

If $u_4\in A(T')$, then it is easy to check that $T\in \mathcal{T}_k$, for $A(T)=A(T')\cup \{u_3\}$, $B(T)=B(T')\cup \{u_2\}$ and $C(T)=C(T')\cup \{u_1\}$.

Finally, we prove that it is not possible to have $u_4\in  B(T')\cup C(T')$ by constructing a $k$-isolating set of size less than $\frac{n+\ell}{2k+1}$, which contradicts that $T$ attains the upper bound. With this aim, let $A_1(T')$ be the vertices of a component of the forest induced by $A(T')$ containing a vertex $a_0$ at distance at most two from $u_4$ (notice that this component is not necessarily unique if $u_4\in B(T')$).
Let $D'$ be the minimum $k$-isolating set considered in Proposition~\ref{prop:isolatingset} and let $D=(D'\setminus \{a_0\})\cup\{u_3\}$. 
By construction,  $D$ is a $k$-isolating set of $T$. 
Indeed, all stars $K_{1,k}$ of $T'$ have a vertex in $N[D']$. 
Since $u_3$ is a neighbour of $u_4$ and $u_2$, we derive that $D$ is also a $k$-isolating set of $T$.
Hence,

{\small $$\iota_k(T)\le |D|=|D'|=\frac{n'+\ell'}{2k+1}=\frac{n-(k+2)+\ell-k}{2k+1}=\frac{n+\ell-2k-2}{2k+1}<\frac{n+\ell}{2k+1}$$}

contradicting that $T$ achieves the upper bound.
\end{enumerate}

This finishes the proof, since we have showed $T\in \mathcal{T}_k$ when $T$ achieves the upper bound.

\end{proof}

The former results allow us to characterize all trees attaining the upper bound proved in Theorem~\ref{th:iotakupperbound}.

\begin{Thm}\label{th:charactiotak}
Let $k\ge 2$. If $T$ is a tree of order $n$ with $\ell$ leaves, then
$\displaystyle\iota_k(T)= \frac{n+\ell}{2k+1}$
 if and only if $T$ is the $k$-star or $T\in \mathcal{T}_k$.
\end{Thm}

\section*{Aknowledgements}
M. Mora is 
partially supported by grant PID2023-150725NB-I00 funded by MICIU/AEI/10.13039/501100011033
 and by grant Gen.Cat.DGR 2021-SGR-00266 from Ag\`encia de Gesti\'o d'Ajuts Universitaris i de Recerca. 
 
 M.J. Souto-Salorio 
 is partially supported 
 by grant
 LATCHING (PID2023-147129OB-C21) funded by MICIU/AEI/10.13039/501100011033 and ERDF, EU
 and by grant
 SCANNER-UDC (PID2020-113230RB-C21) funded by MICIU/AEI/10.13039/501100011033.

\bibliographystyle{abbrv} 
\bibliography{bibfile}

\end{document}